\documentclass[a4paper,12pt]{article}
\usepackage{amsmath, amsfonts, amssymb, amsthm}
\usepackage[english]{babel}

\newcounter{num}[section]

\newenvironment{theorem}
{\refstepcounter{num}%
\bigskip\noindent\nopagebreak[4]{\bf Theorem~\arabic{section}.\arabic{num}. }\it}

\newenvironment{lemma}
{\refstepcounter{num}%
\bigskip\noindent\nopagebreak[4]{\bf Lemma~\arabic{section}.\arabic{num}. }\it}

\newenvironment{definition}
{\refstepcounter{num}%
\bigskip\noindent\nopagebreak[4]{\bf Definition~\arabic{section}.\arabic{num}. }\it}

\newenvironment{remark}
{\refstepcounter{num}%
\bigskip\noindent\nopagebreak[4]{\bf Remark~\arabic{section}.\arabic{num}. }}

\newenvironment{statement}
{\refstepcounter{num}%
\bigskip\noindent\nopagebreak[4]{\bf Statement~\arabic{section}.\arabic{num}. }\it}

\newcommand{\N}{{\mathbb{N}}}

\newcommand{\LL}{{\mathcal{L}}}

\newcommand{\qvar}{{\mathrm{qvar}}}
\newcommand{\pvar}{{\mathrm{pvar}}}

\newcommand{\Rad}{{\mathrm{Rad}}}

\newcommand{\Ss}{{\mathcal{S}}}
\newcommand{\V}{{\mathrm{V}}}
\newcommand{\qq}{{\mathrm{q}_\omega}}
\newcommand{\uu}{{\mathrm{u}_\omega}}
\newcommand{\A}{{\mathcal{A}}}

\newcommand{\Nbf}{{\mathbf{N}}}
\newcommand{\Qbf}{{\mathbf{Q}}}
\newcommand{\Ubf}{{\mathbf{U}}}

\renewcommand{\c}{{\mathbf{c}}}

\newcommand{\C}{{\mathcal{C}}}
\newcommand{\B}{{\mathcal{B}}}

\newcommand{\om}{{\omega}}
\renewcommand{\t}{{\tau}}
\newcommand{\s}{{\sigma}}

\begin{document}
\title{Algebraic geometry over Boolean algebras in the language with constants}
\author{Artem N. Shevlyakov}

\maketitle

\sloppy
\rm

Omsk 644099, Pevtsova st. 13, Omsk Branch of Sobolev Institute of Mathematics, Siberian Branch of the Russian Academy of Sciences, e-mail: a\_shevl@mail.ru, тел. +79081196685.

MSC 06E05

{\bf Abstract.} We study equations over boolean algebras with distinguished elements. We prove the criteria, when a boolean algebra is equationally Noetherian, weakly equationally Noetherian, $\mathbf{q}_\omega$-compact or $\mathbf{u}_\omega$-compact. Also we solve the problem of geometric equivalence in the class of boolean algebras with distinguished elements.


\section*{Introduction}
\setcounter{section}{0}
For an arbitrary algebraic structure (algebra, for shortness) in a language $\LL$ one can define the notion of an equation as a an atomic formula of the language $\LL$. Further, the definition of a solution of an equation can be naturally given. Therefore, it is posed the classification problem of algebraic sets (i.e. sets defined by systems of equations) over an algebra $\A$. One can also define the useful notion of a coordinate algebra over an algebra $\A$. The coordinate algebra is an analog of a coordinate ring in commutative algebra and defines an algebraic set up to isomorphism. In papers by E.~Daniyarova, A.~Miasnikov, V.~Remeslennikov~\cite{uniTh,uniTh_new} it were proved two so-called Unifying Theorems which classify coordinate algebras over an algebra $\A$ with seven equivalent approaches.

The first Unifying Theorem describes the class of all coordinate algebras over a given algebra $\A$, but the second one deals with coordinate algebras of irreducible algebraic sets over $\A$. The unique constraint in the both Unifying Theorems is the Noetherian property of algebra $\A$ (i.e. any system of equation over $\A$ is equivalent to its finite subsystem).

In~\cite{uniTh_III} it was defined three classes of algebras which generalize the class of equationally Noetherian algebras. These classes are called weakly equationally Noetherian, $\uu$-compact and  $\qq$-compact algebras and denoted by $\Nbf^\prime,\Ubf,\Qbf$ respectively. 

Each of the classes $\Nbf^\prime,\Ubf,\Qbf$ inherits some properties of the class of equationally Noetherian algebras $\Nbf$. For example, any algebraic set over an algebra $\A\in\Nbf^\prime$ is defined by a finite system of equations. For every $\uu$-compact algebra $\A$ the both Unifying Theorems remain true, whereas for a $\qq$-compact algebras only the first Unifying Theorem holds.

Thus, any application of Unifying Theorems means the
preliminary solution of the next problem.
\medskip

\noindent{\bf Problem.} Find the classes $\Nbf,\Nbf^\prime,\Ubf,\Qbf$ which contain the given algebra  $\A$.

\medskip

It is known that all classes $\Nbf,\Nbf^\prime,\Qbf,\Ubf$ are pairwise distinct. In~\cite{plotkin} it was defined $\qq$-compact but not equationally Noetherian group. In~\cite{kotov} it was shown that the classes $\Nbf,\Nbf^\prime,\Ubf,\Qbf$ of algebras in the language $\LL=\{f_i^{(1)}|i\in\N\}$ ($f_i^{(1)}$ is an unary function) are pairwise distinct. In~\cite{at_service} we defined the series of semilattices in the language $\LL=\{\wedge\}\cup\{\c_i|i\in\N\}$ with countable many constants. The obtained series shows that the classes $\Nbf,\Nbf^\prime,\Ubf,\Qbf$ of semilattices in the language $\LL$ are pairwise distinct.

An interesting result devoted to the problem above was obtained in~\cite{kotov} and it needs the next  definition.

A system of equations $\Ss$ over an algebra $\A$ is called an $E_k$-{\it system} ($k\in\N$) if $\Ss$ has exactly $k$ solutions in $\A$, however the set of solutions of any finite subsystem $\Ss_0\subseteq\Ss$ is infinite.

\medskip

\noindent{\bf Theorem }\cite{kotov}.
{\it 
Let $\A$ be a $\qq$-compact ($\uu$-compact) algebra, then for any $k\in\{0,1\}$ ($k\in\N$) there does not exist any $E_k$-system over $\A$.
}

\medskip

There are algebras, where the conditions of the theorem above become sufficient for an algebra $\A$ to be $\qq$-compact ($\uu$-compact). In~\cite{for_pinus} it was proved for linearly ordered lattices in the language $\{\wedge,\vee\}\cup\{\c_i|i\in I\}$ extended by constants. In the current paper we prove it for boolean algebras in the language $\LL=\{\vee,\cdot,\bar{ },0,1\}\cup\{\c_i|i\in I\}$, extended by an arbitrary set of constants (the functions $\vee,\cdot,\bar{ }$ means the disjunction, conjunction and negation respectively). 


Let us formulate the main results of the current paper.

\medskip

\noindent{\bf Theorem~\ref{th:q-compactness}} {\it A boolean $\C$-algebra $\B$ is $\qq$-compact iff there are not $E_0$- and $E_1$-system over $\B$.  
}

\medskip

\noindent{\bf Theorem~\ref{th:u-compactness}} {\it A boolean $\C$-algebra $\B$ is $\uu$-compact iff there are not $E_k$-systems over $\B$ for any $k\in\N$.  
}

\medskip

\noindent{\bf Theorem~\ref{th:Noeth_criterion}} {\it A boolean $\C$-algebra $\B$ is equationally Noetherian iff the subalgebra $\C$ generated by the constants of $\LL$ is finite.
}

\medskip

\noindent{\bf Theorem~\ref{th:weakly_Noetherian_criterion}} {\it A boolean $\C$-algebra $\B$ is weakly equationally Noetherian iff the algebra $\C$ of constants is complete in $\B$, i.e. any set of elements $\{\c_j|j\in J\}\subseteq\C$ has the infimum in $\B$ and the infimum  belongs to $\C$.
}

\medskip

Remark that the analog of Criterion~\ref{th:Noeth_criterion} holds for semilattices of the language  $\{\wedge\}\cup\{\c_i|i\in I\}$ extended by an infinite set of constants (see~\cite{at_service}).  

\bigskip

Let us consider the another generalization of the Noetherian property. An algebra $\A$ is  \textit{consistently Noetherian} if any consistent system of equations over $\A$ is equivalent to its finite subsystem. This definition does not equivalent to the Noetherian property in general, since there exists a semilattice of the language $\{\wedge\}\cup\{\c_i|i\in I\}$ which is consistently Noetherian, but there is an inconsistent system of equations $\Ss$, whose all finite subsystems are consistent (see~\cite{at_service}). 

However for boolean algebras of the language $\LL$ we have

\medskip

\noindent {\bf Theorem~\ref{th:noeth_and_noeth_are_equivalent}} {\it If a boolean $\C$-algebra $\B$ is consistently Noetherian then it is equationally Noetherian.

}

\bigskip

Paragraph~\ref{sec:geometric_equivalence} devoted to the problem of geometric equivalence of boolean algebras in the language $\LL$. By definition, boolean algebras $\B_1,\B_2$ of the language $\LL$ are geometrically equivalent if for any system of equations $\Ss$ the coordinate algebras over $\B_1$ and $\B_2$ are isomorphic. It means that the description of coordinate algebras over an algebra $\B_1$ automatically implies the corresponding description over any algebra $\B_2$ which is geometrically equivalent to $\B_1$. 

The problem of geometric equivalence was posed in~\cite{plotkin}. In~\cite{AGoverGroups_2} this problem was solved for equationally Noetherian groups. Theorem~\ref{th:geom_equiv} of the current paper contains a criterion for a pair of boolean algebras of the language $\LL$ to be geometrically equivalent. 

As it follows from~\cite{AGoverGroups_2}, the geometric equivalence of algebras $\B_1,\B_2$ is highly connected to the universal classes generated by the algebras $\B_1,\B_2$. For instance, the equality of the pre-varieties generated by $\B_1,\B_2$ is equivalent to their geometric equivalence. Thus, Theorem~\ref{th:geom_equivalence_all} contains the statements about geometric equivalence and the universal classes generated by boolean algebras of the language $\LL$.

\section{Boolean algebras}

Following~\cite{goncharov,monk}, let us give the main properties of boolean algebras.

Let $\LL_0=\{\vee^{(2)},\cdot^{(2)},\bar{ }^{\ (1)},0,1\}$ be a language with two binary functions $\vee,\cdot$, one unary $\bar{ }$ and constants $0,1$. The functions $\vee,\cdot,\bar{ }$ are called the {\it disjunction, conjunction and negation} respectively. We also use the symbol $\bigwedge$ for the conjunction. For example, the denotation $\bigwedge_{i\in I}b_i$ means the conjunction of the elements $b_i$ with indexes from the set $I$.

The axioms and main identities of boolean algebras can be found in~\cite{goncharov,monk}.

Over a boolean algebra one can define a partial order by
$$x\leq y\Longleftrightarrow x y=x.$$

A boolean algebra $\B$ is {\it complete}  if any set $M\subseteq\B$ has the infimum $\inf M$ and supremum $\sup M$ with respect to the partial order $\leq$. The existence of the infimum for any subset $\{b_i|i\in I\}\subseteq\B$ implies the existence of the supremum of every subset $\{b_i|i\in I\}$ and vise versa. Further, we shall say that a subalgebra $\C\subseteq\B$ is {\it complete in} $\B$ if for an arbitrary set $M\subseteq\C$ one can calculate the infimum $\inf M$ and supremum $\sup M$ in the algebra $\B$, and $\inf M,\sup M\in\C$.  

Let us extend the language $\LL_0$ by infinite set of constants $$\LL=\LL_0\cup\{\c_i|i\in I\}.$$ Any boolean algebra in the language $\LL$ is called a {\it $\C$-algebra}, where $\C$ is the boolean subalgebra generated by the constants $\{\c_i|i\in I\}$.

\section{Basic notions of algebraic geometry}
\label{sec:Foundations_of_Alg_Geom}
In this paragraph we give the main definition of universal algebraic geometry. For more details, see~\cite{uniTh,uniTh_new,uniTh_III}. 

All definition below maybe given for an arbitrary algebra in a language with no predicates. However we give them just for boolean $\C$-algebras.

Let $X=\{x_1,x_2,\ldots,x_n\}$ be a finite set of variables. {\it An equation (or $\C$-equation)} over a boolean $\C$-algebra $\B$ is an atomic formula $\t(X)=\s(X)$ in the language $\LL$. {\it A system of equations (system, for shortness)} is an arbitrary set of equations. 

\begin{remark}
In our paper we always consider systems which depend on a finite set of variables. Secondly, we shall consider the expressions $t(X)\leq s(X)$ as equations, since 
\[
t(X)\leq s(X)\Leftrightarrow t(X)s(X)=t(X).
\]
\end{remark}

The notion of a~\textit{solution set} of an equation $\t(X)=\s(X)$ (a system $\Ss$) over a boolean algebra $\B$ is defined by the natural way and denoted by $\V_\B(\t(X)=\s(X))$ ($\V_\B(\Ss)$). If a system has no solution it is called~\textit{inconsistent}.

A set $Y\subseteq \B^n$ is~\textit{algebraic} over $\B$ if there exists a system $\Ss$ in variables $x_1,x_2,\ldots,x_n$ such that $Y=\V_\B(\Ss)$. 

\textit{The radical} $\Rad_{\B}(\Ss)$ of a system $\Ss$ over a $\C$-algebra $\B$ is the set of all equations $\t(X)=\s(X)$ such that $\V_{\B}(\Ss)\subseteq\V_{\B}(\t(X)=\s(X))$. Obviously, the radical of an inconsistent system contains all equations over a boolean $\C$-algebra $\B$.

Two systems of equations in variables $X$ are~\textit{equivalent} over a boolean $\C$-algebra $\B$, if they have the same solution sets over $\B$. The equivalence of system over a boolean algebra $\B$ is denoted by the symbol $\sim_\B$. If two systems are equivalent over any boolean $\C$-algebra we shall use the symbol $\sim$.

A boolean $\C$-algebra $\B$ is~\textit{equationally Noetherian} if for any system $\Ss$ (even for inconsistent) there exists a finite subsystem $\Ss^\prime\subseteq\Ss$ which is equivalent to $\Ss$ over $\B$. The class of all equational Noetherian boolean $\C$-algebras is denoted by $\mathbf{N}$.

The following four notions generalize the equational Noetherian property.

A boolean $\C$-algebra is
\begin{enumerate}
\item \textit{consistently Noetherian} if for any  system $\Ss$ which is consistent over $\B$ there exists a finite subsystem $\Ss^\prime\subseteq\Ss$ with $\Ss^\prime\sim_{\B}\Ss$. The class of all consistently Noetherian boolean $\C$-algebras is denoted by $\mathbf{N}_c$.
\item \textit{weakly equationally Noetherian} if for every system $\Ss$ there exists a finite system  $\Ss^\prime$ such that $\Ss^\prime\sim_\B\Ss$ (here we do not suppose $\Ss^\prime$ to be a subsystem of $\Ss$). Denote by $\mathbf{N}^\prime$ the class of all weakly equationally Noetherian boolean $\C$-algebras.
\item \textit{$\qq$-compact} if for any system $\Ss$ and an equation $\t(X)=\s(X)$ such that  $\V_\B(\Ss)\subseteq\V_\B(\tau(X)=\sigma(X))$ there exists a finite subsystem $\Ss^\prime\subseteq\Ss$ with
$\V_\B(\Ss^\prime)\subseteq\V_\B(\tau(X)=\sigma(X))$. The class of all $\qq$-compact boolean $\C$-algebras if denoted by $\Qbf$.
\item \textit{$\uu$-compact} if for any system $\Ss$ and a finite set of equations $\tau_i(X)=\sigma_i(X)$ ($1\leq i\leq m$) such that 
\[
\V_\B(\Ss)\subseteq\bigcup_{\substack{i=1}}^m\V_\B(\tau_i(X)=\sigma_i(X)),
\] 
there exists a finite subsystem $\Ss^\prime\subseteq\Ss$ with 
\[
\V_\B(\Ss^\prime)\subseteq\bigcup_{\substack{i=1}}^m\V_\B(\tau_i(X)=\sigma_i(X)).
\] 
Denote by $\Ubf$ the class of all $\uu$-compact boolean $\C$-algebras.
\end{enumerate}

The following picture shows the inclusions of the classes defined above. Notice that  $\mathbf{Q}\cap\mathbf{N}^\prime=\mathbf{N}$ (see the proof in~\cite{uniTh_III}).

\begin{center}
\begin{picture}(150,100)
\qbezier(50,0)(69,0)(85,15) \qbezier(85,15)(100,31)(100,50)
\qbezier(100,50)(100,69)(85,85) \qbezier(85,85)(69,100)(50,100)
\qbezier(50,100)(31,100)(15,85) \qbezier(15,85)(0,69)(0,50)
\qbezier(0,50)(0,31)(15,15) \qbezier(15,15)(31,0)(50,0)

\qbezier(73,93)(70,90)(65,85) 
\qbezier(65,85)(50,69)(50,50)
\qbezier(50,50)(50,31)(65,15) 
\qbezier(65,15)(71,6)(70,4)

\qbezier(100,5)(114,5)(125,15) \qbezier(125,15)(135,26)(135,40)
\qbezier(135,40)(135,54)(125,65) \qbezier(125,65)(114,75)(100,75)
\qbezier(100,75)(86,75)(75,65) \qbezier(75,65)(65,54)(65,40)
\qbezier(65,40)(65,26)(75,15) \qbezier(75,15)(86,5)(100,5)

\put(23,50){$\Qbf$} \put(53,50){$\Ubf$}\put(83,50){$\Nbf$}
\put(113,50){$\Nbf^\prime$}
\end{picture}

\nopagebreak
\textbf{Figure 1.}
\end{center}

For the further study we have a need in the next simple statement whose proof is omitted.

\begin{lemma}
\label{l:aux}
Let $\Ss$ be a system over a boolean $\C$-algebra $\B$ and $M_1,M_2,\ldots,M_m\subseteq\B^n$ be sets 
\[
\V_\B(\Ss)\subseteq\bigcap_{i=1}^mM_i,
\] 
and for any $M_i$ there exists a finite subsystem $\Ss^\prime_i\subseteq\Ss$ with $\V_\B(\Ss^\prime_i)\subseteq M_i$. Then for the finite subsystem $\Ss^\prime=\bigcup_{i=1}^m\Ss^\prime_i$ it holds
\[
\V_\B(\Ss^\prime)\subseteq\bigcap_{i=1}^mM_i
\] 
\end{lemma}
\medskip

Let us give the main definition of our paper.

\begin{definition}\textup{\cite{kotov}}
\label{def:E_k-system}
A system $\Ss$ over a boolean $\C$-algebra $\B$ is called an $E_k$-system if $|\V_\B(\Ss)|=k$, but for any finite subsystem $\Ss^\prime\subseteq\Ss$ it holds $|\V_\B(\Ss^\prime)|=\infty$.  
\end{definition}

\medskip

The next properties of $E_k$-system immediately follows from the definition.

\begin{statement}
Let $\Ss$ be a system over a boolean $\C$-algebra $\B$. Then
\begin{enumerate}
\item if $\Ss$ is not an $E_k$-system and $|\V_\B(\Ss)|=k$, then $\Ss$ is equivalent to its finite subsystem $\Ss^\prime$;
\item in particular, if an inconsistent system $\Ss$ is not an $E_0$-system, there exists a finite inconsistent subsystem $\Ss^\prime\subseteq\Ss$.
\end{enumerate}
\end{statement}

The following theorem was proved in~\cite{kotov} and contains the necessary conditions of $\uu$- and $\qq$-compactness.

\begin{theorem}\textup{\cite{kotov}}
\label{th:kotov}
Let $\B$ be a $\qq$-compact ($\uu$-compact) boolean $\C$-algebra then for $k\in\{0,1\}$ ($k\in\N$) there does not exist an $E_k$-system over $\B$. 
\end{theorem}

\begin{remark}
In the paper~\cite{kotov} Definition~\ref{def:E_k-system} and Theorem~\ref{th:kotov} were formulated for an arbitrary algebraic structure in the language with no predicates.
\end{remark}


Let $\B_1,\B_2$ be boolean algebras of the language $\LL$. We shall say that $\B_1,\B_2$ are {geometrically equivalent} if for any system of $\C$-equations $\Ss$ it holds $\Rad_{\B_1}(\Ss)=\Rad_{\B_2}(\Ss)$.  

\medskip

Let us introduce some definitions devoted to model theory and universal algebra.

A formula of the language $\LL$  
\begin{multline*}
\forall x_1\forall x_2\ldots\forall x_n (\t_1(X)=\s_1(X))\wedge(\t_2(X)=\s_2(X))\wedge\ldots
\\ \wedge(\t_m(X)=\s_m(X))\to(\t(X)=\s(X)),
\end{multline*}
where $\t_i(X),\s_i(X),\t(X),\s(X)$ are terms of $\LL$, is called a \textit{quasi-identity}. The {\it quasivariety} $\qvar(\B)$ generated by a boolean $\C$-algebra $\B$ is the minimal class of all boolean $\C$-algebras which satisfy all quasi-identities $\varphi$ such that $\varphi$ is true in $\B$.   

\textit{The pre-variety} $\pvar(\B)$ generated by a boolean $\C$-algebra is the minimal class of algebras which contains $\B$ and it is closed under the taking subalgebras and Cartesian products.

\begin{theorem}\textup{\textup{\cite{uniTh_III}}}
\label{th:qvar=pvar}
For any $\qq$-compact boolean $\C$-algebra $\B$ it holds
\[
\qvar(\B)=\pvar(\B).
\]  
\end{theorem}    

The next theorem was proved in~\cite{AGoverGroups_2} for groups, however one can prove it for an arbitrary algebraic structure. We give this theorem for boolean $\C$-algebras.

\begin{theorem}\textup{\textup{\cite{AGoverGroups_2}}}
\label{th:geom_equivalence_pvar}
Boolean $\C$-algebras $\B_1,\B_2$ are geometrically equivalent iff
\[
\pvar(\B_1)=\pvar(\B_2).
\]  
\end{theorem}

\section{Transformations of equations}

Let $X=\{x_1,x_2,\ldots,x_n\}$ be a finite set of variables. One can introduce the new variables $Z=\{z_\alpha|\alpha\in\{0,1\}^n\}$ indexed by all $n$-tuples of $0,1$ (hence, $|Z|=2^n$). By $\pi_i(\alpha)$ ($1\leq i\leq n$) we denote the projection of the $n$-tuple $\alpha$ onto $i$-th coordinate. The variables $X=\{x_1,x_2,\ldots,x_n\}$ are replaced to the variables of the set $Z$ by the law

\begin{equation}
\label{eq:old_var_by_new}
x_i=\bigvee_{\pi_i(\alpha)=1}z_\alpha.
\end{equation}

The converse substitution of the variables is defined by the formula
\begin{equation}
\label{eq:new_var_by_old}
z_\alpha=x_1^{a_1}x_2^{a_2}\ldots x_n^{a_2},
\end{equation} 
where $\alpha=(a_1,a_2,\ldots,a_n)$, $a_i\in\{0,1\}$ and
\begin{equation}
x_i^{a_i}=\begin{cases}
x_i \mbox{ if } a_i=1,\\
\overline{x}_i \mbox{ if } a_i=0.
\end{cases}
\end{equation} 
For example, if $X=\{x_1,x_2,x_3\}$ we have $z_{(0,1,1)}=\overline{x}_1x_2x_3$, $z_{(0,0,0)}=\overline{x}_1\overline{x}_2\overline{x}_3$. 

\begin{remark}
Further the disjunctions and conjunctions of the type
\[
\bigvee_{\alpha\in\{0,1\}^n}z_\alpha,\;\bigwedge_{\alpha\in\{0,1\}^n}z_\alpha
\]
are denoted by
\[
\bigvee_{\alpha}z_\alpha,\;\bigwedge_{\alpha}z_\alpha
\]
for shortness.
\end{remark}

Using the axioms of boolean algebra, one can reduce any equation $\t(X)=\s(X)$ to a finite system 
\begin{equation}
\label{eq:this_is_equivalent_to_equation}
\Ss_{\t=\s}=\{z_\alpha\leq\c_\alpha|\alpha\in\{0,1\}^n\}\cup\{z_\alpha z_\beta=0|\alpha\neq\beta\}
\cup\{\bigvee_{\alpha}z_\alpha=1\}
\end{equation} 
in variables $Z$.

From~(\ref{eq:this_is_equivalent_to_equation}) it follows that any system $\Ss(X)$ is equivalent to a system
\begin{eqnarray}
\label{eq:S(Z)}
\Ss=\bigcup_{\alpha}\Ss_\alpha\cup\\
\label{eq:S(Z)_cond1}
\bigcup_{\substack{ \alpha\neq\beta}}\{z_\alpha z_\beta=0\}\cup\\
\label{eq:S(Z)_cond2}
\{\bigvee_{\alpha}z_\alpha=1\}
\end{eqnarray}
where $\Ss_\alpha=\{z_\alpha\leq\c_i|i\in I_\alpha\}$.

\section{Equationally Noetherian boolean algebras}
\label{sec:Noetherian}

\begin{theorem}
\label{th:Noeth_criterion}
A boolean $\C$-algebra $\B$ is equationally Noetherian iff the subalgebra $\C$ generated by the constants of the language $\LL$ is finite.
\end{theorem}
\begin{proof}
Suppose, $\C$ is finite. There exists at most finite number of distinct equations. Hence, all systems over $\B$ are finite. 

Let $\C$ be an infinite algebra. From the theory of boolean algebras it follows there exists an infinite chain 
\[
\c_1<\c_2<\ldots<\c_n<\ldots
\]

It is easy to prove that the system $\Ss=\{x\geq\c_{1},x\geq\c_2,\ldots,x\geq\c_n,\ldots\}$ is not equivalent to any finite subsystem.
\end{proof}

\begin{theorem}
\label{th:noeth_and_noeth_are_equivalent}
If a boolean $\C$-algebra $\B$ is consistently Noetherian then it is equationally Noetherian.
\end{theorem}
\begin{proof}
Without loss of generality one can assume that a system $\Ss$ over $\B$ is written in the form~(\ref{eq:S(Z)},\ref{eq:S(Z)_cond1},\ref{eq:S(Z)_cond2}). Let us prove that the system $\Ss_\alpha=\{z_\alpha\leq\c_i|i\in I_\alpha\}$ is equivalent to its finite subsystem. As any $\Ss_\alpha$ is consistent ($0\in\V_\B(\Ss_\alpha)$), by the condition of the theorem it is equivalent to a finite subsystem $\Ss_\alpha^\prime$. 

Thus, the system $\Ss$ is equivalent to
\[
\Ss^\prime=\bigcup_{\alpha}\Ss_\alpha^\prime\cup\bigcup_{\substack{ \alpha\neq\beta}}\{z_\alpha z_\beta=0\}\cup\{\bigvee_{\alpha}z_\alpha=1\},
\]
and $\B$ is equationally Noetherian.
\end{proof}

\section{Weakly equationally Noetherian boolean algebras}

\label{sec:weak_noetherian}
\begin{theorem}
\label{th:weakly_Noetherian_criterion}
A boolean $\C$-algebra $\B$ is weakly equationally Noetherian iff the algebra $\C$ is complete in $\B$, i.e. any set of elements $\{\c_j|j\in J\}\subseteq\C$ has the infimum in the algebra $\B$ and the infimum  belongs to $\C$.
\end{theorem}
\begin{proof}
If the set of constants $\{\c_j|j\in J\}\subseteq\C$  neither have not the infimum in $\B$ nor its infimum belongs to the set $\B\setminus\C$, then the system $\Ss=\{x\leq\c_j|j\in J\}$ has a nonzero solution in $\B$.

As the boolean $\C$-algebra $\B$ is weakly equationally Noetherian, there exists a finite system 
\[
\Ss_0=\{x\leq\c_j^\prime|j\in J^\prime\}\cup\{\bar{x}\leq\c_j^{\prime\prime}|j\in J^{\prime\prime}\}
\]
which is equivalent to $\Ss$. Obviously, $\Ss_0$ is equivalent to the system of two equations $\{x\leq\c^\prime\}\cup\{\bar{x}\leq\c^{\prime\prime}\}$, where
\[
\c^\prime=\bigwedge_{j\in J^\prime}\c_j^\prime,\;
\c^{\prime\prime}=\bigwedge_{j\in J^{\prime\prime}}\c_j^{\prime\prime}
\]
Since $0\in\B$ is a solution of $\Ss$, we have $0\in\V_\B(\Ss_0)$. Thus, $\c^{\prime\prime}=1$, and the equation $\bar{x}\leq\c^{\prime\prime}$ becomes trivial. It follows that the system $\Ss_0$ is equivalent to $x\leq\c^\prime$.

As $\c^\prime\in\V_\B(\Ss_0)$, we obtain $\c^\prime\in\V_\B(\Ss)$. Therefore, $\c^\prime\leq\c_j$ for all $j\in J$. Since $\c^\prime$ is not the infimum of the set $\{\c_j|j\in J\}$, there exists an element $b\in\B$ such that $\c^\prime<b<\c_j$ for all $j\in J$. 

Indeed, $b\in\V_\B(\Ss)$, however $b\notin\V_\B(\Ss_0)$. Thus, the system $\Ss_0$ is not equivalent to $\Ss$.

Prove the converse. Suppose $\C$ is complete in $\B$, and a system $\Ss$ has the form~(\ref{eq:S(Z)},\ref{eq:S(Z)_cond1},\ref{eq:S(Z)_cond2}). 

Denote by $\c_\alpha\in\C$ the infimum of the set  $\{\c_i|i\in I_\alpha\}$. Obtain that $\Ss$ is equivalent to the finite system
\[
\bigcup_{\alpha}\{z_\alpha\leq\c_\alpha\}\cup\bigcup_{\substack{ \alpha\neq\beta}}\{z_\alpha z_\beta=0\}\cup\{\bigvee_{\alpha}z_\alpha=1\},
\]
and, hence, $\B$ is weakly equationally Noetherian.
\end{proof}

\section{$\qq$- and $\uu$-compactness of boolean algebras}

By $P=(p_\alpha|\alpha\in\{0,1\}^n)$ we denote a point which coordinates $p_\alpha$ are the values of the variables $z_\alpha$.

Suppose the coordinates of the point $P=(p_\alpha|\alpha\in\{0,1\}^n)$ are elements of boolean $\C$-algebra $\B$. 
Consider an arbitrary linear order over the set of indexes $\alpha\in\{0,1\}^n$; the minimal element denote by $\om$ respectively. Define a point $Q=(q_\alpha|\alpha\in\{0,1\}^n)$ as
\begin{equation}
q_\alpha=\begin{cases}
p_{\om}, \; \beta=\om\\
p_\alpha\bigwedge_{\beta<\alpha}\bar{p}_\beta,\; \beta\neq\om
\end{cases}
\label{eq:separation}
\end{equation}

A point $Q$ above is called a {\it splitting of $P$ with the first coordinate $\om$}. The splitting of a point has the following properties.

\begin{lemma}
\label{l:properties_of_splitting}
For the splitting $Q=(q_\alpha|\alpha\in\{0,1\}^n)$ of a point $P=(p_\alpha|\alpha\in\{0,1\}^n)$ the next holds:
\begin{enumerate}
\item $q_\om=p_\om$;
\item $q_\alpha\leq p_\alpha$ for all $\alpha\in\{0,1\}^n$;
\item $q_\alpha q_\gamma=0$, for $\alpha\neq\gamma$;
\item $\bigvee_\alpha p_\alpha=\bigvee_\alpha q_\alpha$.
\end{enumerate}
\end{lemma}
\begin{proof}
The first two statements are immediately follows from the definition of the splitting.

Let us prove the third one. We have 
\[
q_\alpha q_\beta=(p_\alpha\bigwedge_{\gamma<\alpha}\bar{p}_\gamma)(p_\beta\bigwedge_{\delta<\beta}\bar{p}_\delta)
\]

Suppose $\alpha<\beta$ for the given linear order over the set of indexes. Then the conjunction $\bigwedge_{\delta<\beta}\bar{p}_\delta$ contains $\bar{p}_\alpha$, hence 
\[
q_\alpha q_\beta=(p_\alpha\bigwedge_{\gamma<\alpha}\bar{p}_\gamma)(p_\beta \bar{p}_\alpha\bigwedge_{\substack{\delta<\beta\\ \delta\neq\alpha}}\bar{p}_\delta)=p_\alpha \bar{p}_\alpha(\bigwedge_{\gamma<\alpha}\bar{p}_\gamma\bigwedge_{\substack{\delta<\beta\\ \delta\neq\alpha}}\bar{p}_\delta)=0.
\]

Let us prove the fourth statement. Let the indexes $\om$, $\om^\prime$ be minimal and maximal respectively for the given linear order over the set $\{\alpha|\alpha\in\{0,1\}^n\}$. For an arbitrary  $\gamma$ we prove the equality
\begin{equation}
\bigvee_{\alpha\leq\gamma}q_\alpha=\bigvee_{\alpha\leq\gamma}p_\alpha
\label{eq:may_reject1}
\end{equation}
by the induction.

If $\gamma=\om$ we have the true equality $q_{\om}=p_{\om}$. Suppose that for all $\gamma<\delta$ the equality~(\ref{eq:may_reject1}) holds, and prove it for the index $\delta$. Indeed,
\begin{multline*}
\bigvee_{\alpha\leq\delta}q_\alpha=\bigvee_{\alpha<\delta}q_\alpha\vee q_\delta=
\bigvee_{\alpha<\delta}p_\alpha\vee p_\delta\bigwedge_{\beta<\delta}\bar{p}_\beta=
(\bigvee_{\alpha<\delta}p_\alpha\vee p_\delta)\bigwedge_{\beta<\delta}(\bigvee_{\alpha<\delta}p_\alpha\vee \bar{p}_\beta)=\\
(\bigvee_{\alpha\leq\delta}p_\alpha)
\bigwedge_{\beta<\delta}(\bigvee_{\substack{\alpha<\delta \\ \alpha\neq\beta}}p_\alpha\vee p_\beta\vee \bar{p}_\beta)=
(\bigvee_{\alpha\leq\delta}p_\alpha)\bigwedge_{\beta<\delta}(\bigvee_{\substack{\alpha<\delta \\\alpha\neq\beta}}p_\alpha\vee 1)=
\bigvee_{\alpha\leq\delta}p_\alpha.
\end{multline*}

For $\gamma=\om^\prime$ the equality~(\ref{eq:may_reject1}) coincides with the statement of the lemma.

\end{proof}

By Lemma~\ref{l:properties_of_splitting}, we obtain the next simple result about the consistency of systems over boolean algebras. 

\begin{lemma}
\label{l:may_reject}
If the system
\begin{equation}
\Ss(Z)=\bigcup_{\alpha}\Ss_\alpha\\
\label{eq:S(Z)111}
\cup\{\bigvee_{\alpha}z_\alpha=1\}
\end{equation}
is consistent over a boolean $\C$-algebra $\B$, so is the system defined by the equations~(\ref{eq:S(Z)},\ref{eq:S(Z)_cond1},\ref{eq:S(Z)_cond2}). 
\end{lemma}
\begin{proof}
Let $P=(p_\alpha|\alpha\in\{0,1\}^n)$ be a solution of the system~(\ref{eq:S(Z)111}), and $Q=(q_\alpha|\alpha\in\{0,1\}^n)$ is the splitting of $P$ with the first coordinate $\om$. 

Following the third statement of Lemma~\ref{l:properties_of_splitting}, it follows $q_\alpha q_\beta=0$, hence the point $Q$ satisfies the equations~(\ref{eq:S(Z)_cond1}).

For coordinates of the point $P$ it holds $\bigvee_{\alpha}p_\alpha=1$, and by the fourth statement of Lemma~\ref{l:properties_of_splitting} we have $\bigvee_{\alpha}q_\alpha=1$. Therefore, the point $Q$ satisfies the equation~(\ref{eq:S(Z)_cond2}).

By the second statement of Lemma~\ref{l:properties_of_splitting}, we have $q_\alpha\leq p_\alpha$. Hence, all inequalities of the systems $\Ss_\alpha$ hold for the point $Q$.

Thus, the point $Q$ is a solution of the system~(\ref{eq:S(Z)},\ref{eq:S(Z)_cond1},\ref{eq:S(Z)_cond2}).
\end{proof}

\begin{theorem}
\label{th:q-compactness}
A boolean $\C$-algebra $\B$ is $\qq$-compact iff  there do not exist neither $E_0$- nor $E_1$-system over $\B$.  
\end{theorem}
\begin{proof}
The direct statement follows from Theorem~\ref{th:kotov}. Prove the converse.

Let $\Ss$ be a system of equations~(\ref{eq:S(Z)},\ref{eq:S(Z)_cond1},\ref{eq:S(Z)_cond2}) over $\B$ and 
\begin{equation}
\V_\B(\Ss)\subseteq \V_\B(z_\alpha\leq\c)
\label{eq:q-compactness0}
\end{equation} 
for some equation $z_\alpha\leq\c$. 
Below we find a finite subsystem $\Ss^\prime\subseteq\Ss$ with 
\begin{equation}
\label{eq:q-compactness1}
\V_\B(\Ss^\prime)\subseteq \V_\B(z_\alpha\leq\c)
\end{equation}

Suppose $\Ss$ is inconsistent. By the condition, there exists a finite inconsistent subsystem $\Ss^\prime\subseteq\Ss$, and the inclusion~(\ref{eq:q-compactness1}) obviously holds. 

Assume now that a point $P=(p_\beta|\beta\in\{0,1\}^n)$ is a solution of the system $\Ss$. Consider the next system in a single variable $x$
\[
\Ss_0=\{x\leq\c_i|i\in I_\alpha\}\cup\{x\leq \bar{\c}\}.
\]  

There exist exactly two possibilities.
\begin{enumerate}
\item The system $\Ss_0$ has a nonzero solution $x_0\neq 0$. 

Assume $x_0\leq p_\alpha$, and it implies $x_0\leq\c$, as the point $P$ satisfies $z_\alpha\leq \c$. Since the element $x_0$ satisfies the system $\Ss_0$, it holds $x_0\leq \bar{\c}$. However, two inequalities $x_0\leq\c$, $x_0\leq \bar{\c}$ imply $x_0=0$ that contradicts with the choice of the element $x_0$. 

Finally, we have to put $x_0\nleq p_\alpha$.

Let us define a point $Q=(q_\beta|\beta\in\{0,1\}^n)$ as
\[
q_\beta=\begin{cases}
p_\beta,\; \beta\neq\alpha\\
p_\beta\vee x_0,\; \beta=\alpha.
\end{cases}
\]

Clearly, $Q$ satisfies all systems $\Ss_\beta\subseteq\Ss$, $\beta\in\{0,1\}^n$ and the equation~(\ref{eq:S(Z)_cond2}). Denote by $R=(r_\beta|\beta\in\{0,1\}^n)$ the splitting of $Q$ with the first coordinate $\alpha$. By Lemma~\ref{l:may_reject}, the point $R$ is a solution of $\Ss$. As $\alpha$ is the first coordinate of the splitting, we have $r_\alpha=q_\alpha$ and
\[
r_\alpha\c=(p_\alpha\vee x_0)\c=p_\alpha\c\vee x_0\c=p_\alpha\vee 0=p_\alpha.
\]
Since $x_0\nleq p_{\alpha}$, then $p_\alpha\neq r_\alpha$. Thus, $r_\alpha\nleq\c$, however it contradicts with the inclusion~(\ref{eq:q-compactness0}).

\item The system $\Ss_0$ has a unique solution $x=0$. By the condition of the theorem, there is not any $E_1$-system over $\B$, hence $\Ss_0$ is equivalent to its finite subsystem
\[
\Ss_0^\prime=\{x\leq\c_i|i\in I_\alpha^\prime\}\cup\{x\leq \bar{\c}\},\; |I_\alpha^\prime|<\infty.
\]  

The equality $\V_\B(\Ss_0^\prime)=\{0\}$ implies
\[
\bar{\c}\cdot\bigwedge_{i\in I_\alpha^\prime}\c_i=0\Leftrightarrow\bigwedge_{i\in I_\alpha^\prime}\c_i\leq\c.
\]

Thus, for the finite subsystem $\Ss^\prime=\{z_\alpha\leq\c_i|i\in I_\alpha^\prime\}\subseteq\Ss_\alpha\subseteq\Ss$ the inclusion~(\ref{eq:q-compactness1}) holds.
\end{enumerate}

We proved that for the equations of the form $z_\alpha\leq\c$ there exists a finite subsystem $\Ss^\prime$ which makes the inclusion~(\ref{eq:q-compactness1}) valid. Consider now an equation $\t(X)=\s(X)$ in variables $X=\{x_1,x_2,\ldots,x_n\}$ such that the next inclusion holds
\[
\V_\B(\Ss)\subseteq \V_\B(\t(X)=\s(X)).
\]
The equation $\t(X)=\s(X)$ can be written as a system $\Ss_{\t=\s}$~(\ref{eq:this_is_equivalent_to_equation}). Above we proved that for any equation $z_\alpha\leq\c_\alpha\in\Ss_{\t=\s}$ there exists a finite subsystem $\Ss^\prime_\alpha\subseteq\Ss_\alpha$ with
\[
\V_\B(\Ss^\prime_\alpha)\subseteq\V_\B(z_\alpha\leq\c_\alpha).
\] 
Finally, by Lemma~\ref{l:aux}, for the finite subsystem $\Ss^\prime=\bigcap_{\alpha}\Ss^\prime_\alpha\subseteq\Ss$ it holds 
\[
\V_\B(\Ss^\prime)\subseteq \V_\B(\t(X)=\s(X)).
\]
\end{proof}

\begin{theorem}
\label{th:u-compactness}
A boolean $\C$-algebra $\B$ is $\uu$-compact iff there are not $E_k$-systems over $\B$ for any $k\in\N$.  
\end{theorem}
\begin{proof}
The direct statement was proven in Theorem~\ref{th:kotov}. Prove the converse.

Let $\Ss$ be a system of equations~(\ref{eq:S(Z)},\ref{eq:S(Z)_cond1},\ref{eq:S(Z)_cond2}) over $\B$ and 
\begin{equation}
\V_\B(\Ss)\subseteq \bigcup_{\alpha}\V_\B(z_\alpha\leq\c_\alpha).
\label{eq:u-compactness0}
\end{equation} 
Let us define a finite subsystem $\Ss^\prime\subseteq\Ss$ with 
\begin{equation}
\label{eq:u-compactness1}
\V_\B(\Ss^\prime)\subseteq \bigcup_{\alpha}\V_\B(z_\alpha\leq\c_\alpha)
\end{equation}

By the condition of the theorem, for the inconsistent system $\Ss$ there exists a finite subsystem $\Ss^\prime\subseteq\Ss$ which satisfies the inclusion~(\ref{eq:u-compactness1}). 

Consider a solution $P=(p_\beta|\beta\in\{0,1\}^n)$ of the system $\Ss$. Define the next systems ($\alpha\in\{0,1\}^n$) in variable $x$
\[
\Ss_{0\alpha}=\{x\leq\c_i|i\in I_\alpha\}\cup\{x\leq \bar{\c}_\alpha\}.
\]  

We have exactly two cases.
\begin{enumerate}
\item Any system $\Ss_{0\alpha}$ has a nonzero solution $x_{0\alpha}\neq 0$. As $x_{0\alpha}\leq \bar{\c}_\alpha$, then $x_{0\alpha}\c_\alpha=0$. One can assume that all systems $\Ss_{0\alpha}$ are not equivalent to their finite subsystems (otherwise, one can replace some systems to its equivalent finite subsystems and further deal only with the unchanged systems). 

By the condition of the theorem, any system $\Ss_{0\alpha}$ has an infinite solution set. 

Further we shall define a point $Q=(q_\alpha|\alpha\in\{0,1\}^n)$ such that
\begin{enumerate}
\item $q_\alpha\in\V_\B(\Ss_{0 \alpha})\setminus\{0\}$;
\item $q_\alpha q_\beta=0$;
\item $q_\alpha x_{0\alpha}\neq 0$.
\end{enumerate}

We define the point $Q$ by the induction on the number of the systems $\Ss_{0 \alpha}$. For the convenience we denote the indexes of the systems $\Ss_{0\alpha}$ by the natural numbers $1,2,\ldots,m$. 

{\bf The basis.} For a single system $\Ss_{1}$ put $q_\alpha=x_{0 1}$.

{\bf The inductive hypothesis.} Assume that for the systems $\Ss_{i}$, $1\leq i\leq m-1$ there exists a point $Q^\prime=(q_i^\prime|1\leq i\leq m-1)$ with the properties above. 

{\bf The inductive step.} As the system $\Ss_m$ has an infinite solution set, by the element $x_{0m}$ one can choose a solution of $\Ss_m$ such that 
\[
x_{0m}\neq\bigvee_{i\in K} q^\prime_i
\]
for any set of indexes $K\subseteq\{1,2,\ldots,m-1\}$.

Consider the next cases.
\begin{enumerate}
\item $\bigvee_{i=1}^{m-1} q^\prime_i\ngeq x_{0m}$. Hence, the element  $x_{0m}\overline{\bigvee_{i=1}^{m-1} q^\prime_i}$ does not equal $0$. A point $Q$ defined as
\[
q_j=\begin{cases}
q_j^\prime,\;1\leq j\leq m-1\\
x_{0m}\overline{\bigvee_{i=1}^k q^\prime_i},\; j=m
\end{cases}
\]
obviously satisfies all properties above.

\item $\bigvee_{i=1}^{m-1} q^\prime_i>x_{0m}$. Without loss of generality, one can assume that for any $1\leq i\leq m-1$ it holds $q^\prime_ix_{0m}\neq 0$ (otherwise, if $q^\prime_ix_{0m}=0$ for some $q^\prime_i$ we exclude $q^\prime_i$ from the disjunction $\bigvee_{i=1}^{m-1} q^\prime_i$). 

For some number $1\leq i_0\leq m-1$ it holds $q^\prime_{i_0} x_{0m}<q^\prime_{i_0}$ (if for all $i$ $q^\prime_ix_{om}=q^\prime_{i}$, then $\bigvee_{i=1}^{m-1} q^\prime_i\leq x_{0m}$ that contradicts with the condition). From the inequality $q^\prime_{i_0} x_{0m}<q^\prime_{i_0}$ it follows $q^\prime_{i_0} \bar{x}_{0m}\neq 0$. It is directly checked that $Q$ with coordinates
\[
q_j=\begin{cases}
q_j^\prime,\;1\leq j\leq m-1,j\neq i_0\\
q^\prime_{i_0} \bar{x}_{0m},\; j=i_0\\
q^\prime_{i_0} x_{0m},\; j=m
\end{cases}
\]
satisfies all properties above.
\end{enumerate}
\bigskip 

Finally, the point $Q$ with required properties is found. Remark that the inequality  $q_\alpha\leq\bar{\c}_\alpha$ implies $q_\alpha\c_\alpha=0$ for all $\alpha\in\{0,1\}^n$. 

Suppose coordinates of a point $R=(r_\alpha|\alpha\in\{0,1\}^n)$ are defined by
\[
r_\alpha=p_\alpha\bigwedge_{\beta}\bar{q}_\beta\vee q_\alpha.
\]

Since $p_\alpha\in\V_\B(\Ss)$ and $q_\alpha\in\V_\B(\Ss_\alpha)$, hence $r_\alpha\in\V_\B(\Ss_\alpha)$ for every $\alpha\in\{0,1\}^n$.

For $\alpha\neq\gamma$ we have 
\begin{multline*}
r_\alpha r_\gamma=(p_\alpha \bigwedge_{\beta}\bar{q}_\beta\vee q_\alpha)
(p_\gamma \bigwedge_{\beta}\bar{q}_\beta\vee q_\gamma)=\\
p_\alpha p_\gamma\bigwedge_{\beta}\bar{q}_\beta\vee p_\alpha [\bigwedge_{\beta}\bar{q}_\beta q_\gamma]\vee p_\gamma[q_\alpha  \bigwedge_{\beta}\bar{q}_\beta]\vee q_\alpha q_\gamma=0\vee 0\vee 0\vee 0=0.
\end{multline*}
Thus, the point $R$ satisfies the equations~(\ref{eq:S(Z)_cond1}).

Compute 
\begin{multline*}
\bigvee_{\alpha}r_\alpha=\bigvee_{\alpha}(p_\alpha \bigwedge_{\beta}\bar{q}_\beta\vee q_\alpha)=
\bigwedge_{\beta}\bar{q}_\beta\bigvee_{\alpha}p_\alpha\vee \bigvee_{\alpha}q_\alpha=
\bigwedge_{\beta}\bar{q}_\beta\cdot 1\vee \bigvee_{\alpha}q_\alpha=\\
\bigwedge_{\beta}\bar{q}_\beta\vee \bigvee_{\alpha}q_\alpha=
\bigwedge_{\beta}(\bar{q}_\beta\vee\bigvee_{\alpha}q_\alpha)=
\bigwedge_{\beta}(\bar{q}_\beta\vee q_\beta\vee\bigvee_{\alpha\neq\beta}q_\alpha)=\bigwedge_{\beta}1=1,
\end{multline*}
hence the point $R$ satisfies the equation~(\ref{eq:S(Z)_cond2}), and we finally obtain $R\in\V_\B(\Ss)$. 

By the inclusion~(\ref{eq:u-compactness0}), there exists an index $\alpha$ with $r_\alpha\leq\c_\alpha$. We have 
\[
r_\alpha\c_\alpha=(p_\alpha\bigwedge_{\beta}\bar{q}_\beta\vee q_\alpha)\c_\alpha=p_\alpha\bigwedge_{\beta}\bar{q}_\beta\c_\alpha\vee 0=p_\alpha\bigwedge_{\beta}\bar{q}_\beta\c_\alpha.
\]
As $r_\alpha\c_\alpha=r_\alpha$, the next equality holds
\begin{equation}
p_\alpha\bigwedge_{\beta}\bar{q}_\beta\c_\alpha=p_\alpha\bigwedge_{\beta}\bar{q}_\beta\vee q_\alpha.
\label{eq:uuuuuuuuuuu}
\end{equation}

As $p_\alpha\bigwedge_{\beta}\bar{q}_\beta\c_\alpha\leq p_\alpha\bigwedge_{\beta}\bar{q}_\beta$, the equality~(\ref{eq:uuuuuuuuuuu}) implies
\begin{eqnarray}
p_\alpha\bigwedge_{\beta}\bar{q}_\beta\c_\alpha=p_\alpha\bigwedge_{\beta}\bar{q}_\beta, \nonumber\\
\label{eq:uuuuu1}
q_\alpha\leq p_\alpha\bigwedge_{\beta}\bar{q}_\beta\c_\alpha.
\end{eqnarray}

The last inequality is impossible, since the expression $\bigwedge_{\beta}\bar{q}_\beta$ contains $\bar{q}_\alpha$, hence, $q_\alpha (p_\alpha\bigwedge_{\beta}\bar{q}_\beta\c_\alpha)=0\neq q_\alpha$.

Thus, we have that the point $R\in\V_\B(\Ss)$ does not satisfy any inequality $z_\alpha\leq\c_\alpha$ that contradicts with the inclusion~(\ref{eq:u-compactness0})

\item There exists a system $\Ss_{0\beta}$ with $\V_\B(\Ss_{0\beta})=\{0\}$. By the condition of the theorem, $\Ss_{0\beta}$ is equivalent to its finite subsystem
\[
\Ss_{0\beta}^\prime=\{x\leq\c_i|i\in I_\beta^\prime\}\cup\{x\leq \bar{\c}_\beta\},\; |I_\beta^\prime|<\infty.
\]  

The equality $\V_\B(\Ss_{0\beta}^\prime)=\{0\}$ gives 
\[
\bar{\c}_\beta\bigwedge_{i\in I_\beta^\prime}\c_i=0\Leftrightarrow\bigwedge_{i\in I_\beta^\prime}\c_i\leq\c_\beta.
\]

From the last formula it follows that the finite subsystem $\Ss^\prime=\{z_\alpha\leq\c_i|i\in I_\beta^\prime\}\subseteq\Ss_\beta\subseteq\Ss$ satisfies the inclusion~(\ref{eq:u-compactness1}), and the theorem is proved.
\end{enumerate}

\bigskip

Consider now the case, when the right part of the inclusion~(\ref{eq:u-compactness0}) contains more than one equations $\{z_\alpha\leq\c_{i\alpha}|1\leq i\leq n_\alpha\}$ in variable $z_\alpha$. Assume that the next inclusion holds
\begin{equation}
\V_\B(\Ss)\subseteq \bigcup_{\alpha}\bigcup_{i=1}^{n_\alpha}\V_\B(z_\alpha\leq\c_{i\alpha}).
\label{eq:u-compactness2}
\end{equation}

Below we shall define a finite subsystem $\Ss^\prime\subseteq\Ss$ such that  
\begin{equation}
\label{eq:u-compactness3}
\V_\B(\Ss^\prime)\subseteq \bigcup_{\alpha}\bigcup_{i=1}^{n_\alpha}\V_\B(z_\alpha\leq\c_{i\alpha}).
\end{equation}

Let $\tilde{\Ss}$ be an auxiliary system defined by the next way. Let the systems $\Ss_{i\alpha}$ ($1\leq i\leq n_\alpha$) are obtained from $\Ss_\alpha$ by the substitution of the variable $z_\alpha$ to the new variable $z_{i\alpha}$. Hence,

\[
\tilde{\Ss}=\bigcup_{\alpha}\bigcup_{i=1}^{n_\alpha}\Ss_{i\alpha}\cup\bigcup_{\substack{ \alpha\neq\beta \\ i\neq j}}\{z_{i\alpha} z_{j\beta}=0\}\cup\{\bigvee_{\alpha}\bigvee_{i=1}^{n_\alpha}z_{i\alpha}=1\}.
\]

Let us show the inclusion
\begin{equation}
\V_\B(\tilde{\Ss})\subseteq \bigcup_{\alpha}\bigcup_{i=1}^{n_\alpha}\V_\B(z_{i\alpha}\leq\c_{i\alpha}).
\label{eq:u-compactness4}
\end{equation}

Assume the converse, i.e. there exists a point $\tilde{P}=(\tilde{p}_{i\alpha}|\alpha\in\{0,1\}^n,1\leq i\leq n_\alpha)\in\V_\B(\tilde{\Ss})$ such that 
\begin{equation}
\label{eq:u-compactness44}
\tilde{p}_{i\alpha}\nleq\c_{i\alpha}
\end{equation} for all $\alpha\in\{0,1\}^n,1\leq i\leq n_\alpha$.

It is directly checked that the point $P=(p_\alpha|\alpha\in\{0,1\}^n)$, where
\[
p_\alpha=\bigvee_{i=1}^{n_\alpha}\tilde{p}_{i\alpha},
\]  
is a solution of $\Ss$. By~(\ref{eq:u-compactness2}), there exists $\beta\in\{0,1\}^n$ and a number $1\leq j\leq n_{\beta}$ such that $p_\beta\leq\c_{j\beta}$. Hence,
\[
\bigvee_{i=1}^{n_\beta}\tilde{p}_{i\beta}\leq\c_{j\beta}\Rightarrow
\tilde{p}_{i\beta}\leq\c_{j\beta}
\]
for all $1\leq i\leq n_\beta$. However, the last inequality contradicts with~(\ref{eq:u-compactness44}).

Thus, for the system $\tilde{\Ss}$ we have the inclusion~(\ref{eq:u-compactness4}) which has the form~(\ref{eq:u-compactness0}). Hence, for $\tilde{\Ss}$ one can repeat all reasonings above, and obtain a finite subsystem $\tilde{\Ss}^\prime\subseteq\tilde{\Ss}$ with

\begin{equation}
\V_\B(\tilde{\Ss}^\prime)\subseteq \bigcup_{\alpha}\bigcup_{i=1}^{n_\alpha}\V_\B(z_{i\alpha}\leq\c_{i\alpha}).
\label{eq:u-compactness5}
\end{equation}

Let $\Ss^\prime$ be a system in variables $\{z_\alpha|\alpha\in\{0,1\}^n\}$ obtained from $\tilde{\Ss}^\prime$ by the substitution of all variables $z_{i\alpha}$ ($1\leq i\leq n_\alpha$) to $z_\alpha$. 

It is easy to check that the system $\Ss^\prime$ satisfies the inclusion~(\ref{eq:u-compactness3})

\bigskip

Consider now the most general type of equations which occur in the right part of the inclusion~(\ref{eq:u-compactness0}).

Suppose the equations $\t_i(X)=\s_i(X)$, $1\leq i\leq m$ in variables $X=\{x_1,x_2,\ldots,x_n\}$ satisfy the inclusion
\begin{equation}
\V_\B(\Ss)\subseteq \bigcup_{i=1}^m\V_\B(\t_i(X)=\s_i(X)).
\label{eq:u-compactness6}
\end{equation} 

Replace the variables of the set $X$ to $Z=\{z_\alpha|\alpha\in\{0,1\}^n\}$, and obtain the inclusion
\begin{equation}
\V_\B(\Ss)\subseteq \bigcup_{i=1}^m\V_\B(\bar{\Ss}_i),
\label{eq:u-compactness66}
\end{equation} 
where a finite system $\bar{\Ss}_i=\{z_\alpha\leq\c_{\alpha i}|\alpha\in\{0,1\}^n\}$ of the form~(\ref{eq:this_is_equivalent_to_equation}) was obtained from $\t_i(X)=\s_i(X)$ by the substitutions~(\ref{eq:old_var_by_new}). Remark that we do not include in the systems $\bar{\Ss}_i$ the equations~(\ref{eq:S(Z)_cond1},\ref{eq:S(Z)_cond2}), since such equations belong to $\Ss$.

According the distributivity of the union of sets, the union $\bigcup_{i=1}^m\V_\B(\bar{\Ss}_i)$ is a finite intersection of the sets
\[
M(\alpha_1,\alpha_2,\ldots,\alpha_m)=V_\B(z_{\alpha_1}\leq\c_{\alpha_1 })\cup\V_\B(z_{\alpha_2}\leq\c_{\alpha_2})\cup\ldots\cup\V_\B(z_{\alpha_m}\leq\c_{\alpha_m}).
\]

Thus, we have the inclusion 
\[
\V_\B(\Ss)\subseteq\bigcap_{\alpha_1,\alpha_2,\ldots,\alpha_m} M(\alpha_1,\alpha_2,\ldots,\alpha_m).
\]
Above we proved that for any set $M(\alpha_1,\alpha_2,\ldots,\alpha_m)$ there exists a finite subsystem $\Ss^\prime(\alpha_1,\alpha_2,\ldots,\alpha_m)$. By Lemma~\ref{l:aux}, for the system $$\Ss^\prime=\bigcup_{\alpha_1,\alpha_2,\ldots,\alpha_m}\Ss^\prime(\alpha_1,\alpha_2,\ldots,\alpha_m),$$  it holds
\begin{equation*}
\V_\B(\Ss^\prime)\subseteq \bigcup_{i=1}^m\V_\B(\t_i(X)=\s_i(X)).
\end{equation*} 

\end{proof}

\section{Geometric equivalence of boolean algebras}
\label{sec:geometric_equivalence}

\begin{remark}
Let $\B_1,\B_2$ be boolean algebras in the language $\LL$, and the constants from $\LL$ generates the subalgebras $\C_i\subseteq\B_i$  (in other words, $\B_i$ is a boolean $\C_i$-algebra). From the paper~\cite{uniTh} it is easy to prove that the geometric equivalence of $\B_1,\B_2$ implies th isomorphism between $\C_1$ and $\C_2$. Thus, {\bf all  boolean algebras considered below have the isomorphic subalgebras generated by constants}. Equivalently, any boolean algebras $\B_1,\B_2$ in this paragraph have the same subalgebra $\C$ generated by the constants of the language $\LL$. 
\end{remark}

\begin{theorem}
\label{th:geom_equiv}
Two boolean $\C$-algebras $\B_1,\B_2$ are geometrically equivalent iff the following conditions holds: 
\begin{enumerate}
\item any inconsistent over $\B_1$ system of equations is inconsistent over $\B_2$ and vise versa;
\item the infimum of any set of constants $\{\c_j|j\in J\}\subseteq\C$ exists and equals $0$ in $\B_1$ iff the infimum of the same set exists and equals $0$ in $\B_2$.
\end{enumerate}
\end{theorem}
\begin{proof}
If boolean $\C$-algebras $\B_1,\B_2$ are geometrically equivalent then for any system of $\C$-equations $\Ss$ it holds $\Rad_{\B_1}(\Ss)=\Rad_{\B_2}(\Ss)$. If $\Ss$ is inconsistent over $\B_1$ its radical coincides with the set of all $\C$-equations. Hence, the radical $\Rad_{\B_2}(\Ss)$ is also contains all $\C$-equations, and the system $\Ss$ is inconsistent over $\B_2$. 

If the infimum of the set $\{\c_j|j\in J\}$ exists and equals $0$ in $\B_1$, the system $\Ss=\{x\leq\c_j|j\in J\}$ has a unique solution $x=0$ over $\B_1$. Therefore, the equation $x=0$ belongs to the radical $\Rad_{\B_1}(\Ss)$. By geometric equivalence, we have $x=0\in\Rad_{\B_2}(\Ss)$. Thus,  $\Ss$ has a unique solution $x=0$ over $\B_2$, and the infimum of the set $\{\c_j|j\in J\}$ equals $0$ in $\B_2$.

Prove the converse. Let $\Ss$ be a system of $\C$-equations. Show that for any equation $\t(X)=\s(X)$ such that $\t(X)=\s(X)\in\Rad_{\B_1}(\Ss)$ it follows $\t(X)=\s(X)\in\Rad_{\B_2}(\Ss)$.

Write the system $\Ss$ in the form~(\ref{eq:S(Z)},\ref{eq:S(Z)_cond1},\ref{eq:S(Z)_cond2}). An equation $\t(X)=\s(X)$ is equivalent to the system $\Ss_{\t=\s}$~(\ref{eq:this_is_equivalent_to_equation}). It is sufficient to prove that any equation $z_\alpha\leq\c_\alpha$ of $\Ss_{\t=\s}$ belongs to the radical $\Rad_{\B_2}(\Ss)$. 

For an equation $z_\alpha\leq\c$ of $\Ss_{\t=\s}$ let us define the auxiliary system
\begin{equation}
\label{eq:geom_equiv3}
{\Ss_0}=\{x\leq\c_i|i\in I_\alpha\}\cup\{x\leq \bar{\c}_\alpha\}.
\end{equation}

We have two cases.
\begin{enumerate}
\item $\V_{\B_1}({\Ss_0})=\{0\}$.

Therefore, the infimum of the set $\{\c_i|i\in I_\alpha\}\cup\{\bar{\c}\}$ equals $0$ in $\B_1$. By the second condition of the theorem, we have $\V_{\B_2}({\Ss_0})=\{0\}$.

Assume there is a point $P=(p_\beta|\beta\in\{0,1\}^n)\in\V_{\B_2}(\Ss)$ such that $p_\alpha\nleq\c_\alpha$. Consider the element $z_0=p_\alpha \bar{\c}_\alpha$. It is directly checked that $z_0$ is the solution of the system~(\ref{eq:geom_equiv3}) over $\B_2$. As $p_\alpha\nleq\c_\alpha$, then $z_0\neq 0$ that contradicts with  $\V_{\B_2}({\Ss_0})=\{0\}$.  Thus, for all solutions $P=(p_\alpha|\alpha\in\{0,1\}^n)$ of the system $\Ss$ it holds $p_\alpha\leq\c_\alpha$. Finally, $z_\alpha\leq\c_\alpha\in\Rad_{\B_2}(\Ss)$

\item Suppose that ${\Ss_0}$ has a nonzero solution $x_0\in\B_1$. Let $P=(p_\alpha|\alpha\in\{0,1\}^n)$ be a solution of $\Ss$ over $\B_1$ (if $\Ss$ is inconsistent over $\B_1$, by the first condition of the lemma, $\Ss$ is consistent over $\B_2$, and it holds $\Rad_{\B_1}(\Ss)=\Rad_{\B_2}(\Ss)$).  
If we assume $x_0\leq p_\alpha$ then $x_0\leq\c_\alpha$, since $P$ satisfies the equation $z_\alpha\leq \c_\alpha$. As the element $x_0$ satisfies the system $\Ss_0$, we have $x_0\leq \bar{\c}_\alpha$. However, the inequalities $x_0\leq\c_\alpha$ и $x_0\leq \bar{\c}_\alpha$ imply $x_0=0$ that contradicts with the choice of the element $x_0$. 

Finally, we have obtained $x_0\nleq p_\alpha$.

Define a point $Q=(q_\beta|\beta\in\{0,1\}^n)$ by
\[
q_\beta=\begin{cases}
p_\beta,\; \beta\neq\alpha\\
p_\beta\vee x_0,\; \beta=\alpha.
\end{cases}
\]

Obviously, $Q$ satisfies all systems $\Ss_\beta\subseteq\Ss$, $\beta\in\{0,1\}^n$, and $Q$ is a solution of the equation~(\ref{eq:S(Z)_cond2}). Denote by $R=(r_\beta|\beta\in\{0,1\}^n)$ the splitting of $Q$ with the first coordinate $\alpha$. According Lemma~\ref{l:may_reject}, the point $R$ is a solution of $\Ss$. As $\alpha$ is the first coordinate of the splitting, we have $r_\alpha=q_\alpha$, and
\[
r_\alpha\c_\alpha=(p_\alpha\vee x_0)\c_\alpha=p_\alpha\c\vee x_0\c_\alpha=p_\alpha\vee 0=p_\alpha\neq r_\alpha.
\]
Hence, $r_\alpha\nleq\c_\alpha$, and we obtain
\[
\V_{\B_1}(\Ss)\nsubseteq\V_{\B_1}(z_\alpha\leq\c).
\]
The last formula gives us the contradiction $z_\alpha\leq\c_\alpha\notin\Rad_{\B_1}(\Ss)$.
\end{enumerate}

\end{proof}

The next theorem contains all main results devoted to geometric equivalence of boolean $\C$-algebras.

\begin{theorem}
\label{th:geom_equivalence_all}
Let $\B_1,\B_2$ be boolean $\C$-algebras. The following statements hold:
\begin{enumerate}
\item for any finite system $\Ss$ we have
\[
\Rad_{\B_1}(\Ss)=\Rad_{\B_2}(\Ss);
\]
\item $\qvar(\B_1)=\qvar(\B_2)$;
\item if $\B_1,\B_2$ are weakly equationally Noetherian they are geometrically equivalent;
\item if $\B_1,\B_2$ are $\qq$-compact they are geometrically equivalent.

\end{enumerate} 
\end{theorem}
\begin{proof}

\begin{enumerate}
\item It is sufficient to consider a finite system $\Ss$ defined by~(\ref{eq:S(Z)},\ref{eq:S(Z)_cond1},\ref{eq:S(Z)_cond2}) and an $z_\gamma\leq\c\in\Rad_{\B_1}(\Ss)$. Prove that $z_\gamma\leq\c\in\Rad_{\B_2}(\Ss)$

A finite subsystem  $\Ss_\alpha$ of $\Ss$ is equivalent to the equation $z_\alpha\leq\c_\alpha$, where $\c_\alpha=\bigwedge_{i\in I_\alpha}\c_i$. Therefore, the system $\Ss$ can be written in the form
\[
\Ss=\bigcup_{\alpha}\{z_\alpha\leq\c_\alpha\}\cup\bigcup_{\substack{ \alpha\neq\beta}}\{z_\alpha z_\beta=0\}\cup\{\bigvee_{\alpha}z_\alpha=1\}
\]

If $\bigvee_{\alpha}\c_\alpha<1$ the system $\Ss$ is inconsistent over the both boolean $\C$-algebras $\B_1,\B_2$, and the equality $\Rad_{\B_1}(\Ss)=\Rad_{\B_2}(\Ss)$ is obviously holds.

Suppose now that $\bigvee_{\alpha}\c_\alpha=1$. It is easy to see that the point $P=(\c_\alpha|\alpha\in\{0,1\}^n)$ is a solution of 
\[
\bigcup_{\alpha}\{z_\alpha\leq\c_\alpha\}\cup\{\bigvee_{\alpha}z_\alpha=1\}
\]
over the algebras $\B_1,\B_2$.

The splitting $Q$ of the point $P$ with the first coordinate $\gamma$ is a solution of $\Ss$ over the algebras $\B_1,\B_2$. By the definition of the splitting, the coordinates of the index $\gamma$ in the points $P,Q$ equal $\c_\gamma$. As $z_\gamma\leq\c\in\Rad_{\B_1}(\Ss)$, for the algebra $\C$ we have $\c_\gamma\leq\c$, and $z_\gamma\leq\c\in\Rad_{\B_2}(\Ss)$.

\item Let 
\begin{multline*}
\varphi\colon \forall x_1\forall x_2\ldots\forall x_n (\t_1(X)=\s_1(X))\wedge(\t_2(X)=\s_2(X))\wedge\ldots
\\ \wedge(\t_m(X)=\s_m(X))\to(\t(X)=\s(X))
\end{multline*}
be an arbitrary quasi-identity which is not true in boolean $\C$-algebra $\B_1$. It means that there exist elements $b_1,b_2,\ldots,b_n\in\B_1$ such that $\t_i(b_1,b_2,\ldots,b_n)=\s_i(b_1,b_2,\ldots,b_n)$ for all $1\leq i\leq m$, but $\t(b_1,b_2,\ldots,b_n)\neq\s(b_1,b_2,\ldots,b_n)$. Hence, the equation $\t(X)=\s(X)$ does not belong to the radical $\Rad_{\B_1}(\{\t_i(X)=\s_i(X)|1\leq i\leq m\})$. By the first statement of the theorem we have $\t(X)=\s(X)\notin\Rad_{\B_2}(\{\t_i(X)=\s_i(X)|1\leq i\leq m\})$, hence there exist elements $b_1^\prime,b_2^\prime,\ldots,b_n^\prime\in\B_2$ such that  $\t_i(b_1^\prime,b_2^\prime,\ldots,b_n^\prime)=\s_i(b_1^\prime,b_2^\prime,\ldots,b_n^\prime)$ for any $1\leq i\leq m$, but $\t(b_1^\prime,b_2^\prime,\ldots,b_n^\prime)\neq\s(b_1^\prime,b_2^\prime,\ldots,b_n^\prime)$. We obtain that the quasi-identity $\varphi$ is not true in $\B_2$. Thus, $\qvar(\B_1)=\qvar(\B_2)$. 

\item Consider a system of equations $\Ss$ of the form~(\ref{eq:S(Z)},\ref{eq:S(Z)_cond1},\ref{eq:S(Z)_cond2}). By Theorem~\ref{th:weakly_Noetherian_criterion}, the algebra $\C$ is complete in $\B_1,\B_2$. In other words, for the set of constants $\{\c_i|i\in I_\alpha\}$ there exists the infimum $\c_\alpha\in\C$. Hence, the system $\Ss_\alpha$ is equivalent over $\B_1,\B_2$ to the equation $z_\alpha\leq\c_\alpha$.

Thus, the system $\Ss$ is equivalent over the both algebras $\B_1,\B_2$ to a finite system $\Ss^\prime$
\[
\Ss^\prime=\bigcup_{\alpha}\{z_\alpha\leq\c_\alpha\}\cup\bigcup_{\substack{ \alpha\neq\beta}}\{z_\alpha z_\beta=0\}\cup\{\bigvee_{\alpha}z_\alpha=1\}.
\]

Following the first statement of this theorem, we obtain the equality   $\Rad_{\B_1}(\Ss^\prime)=\Rad_{\B_1}(\Ss^\prime)$, and the algebras $\B_1,\B_2$ are geometrically equivalent.

\item By the second statement of the theorem, we have the equality $\qvar(\B_1)=\qvar(\B_2)$. According Theorem~\ref{th:qvar=pvar}, $\pvar(\B_1)=\pvar(\B_2)$. Use Theorem~\ref{th:geom_equivalence_pvar} and obtain the geometric equivalence of boolean $\C$-algebras $\B_1,\B_2$.

\end{enumerate}
\end{proof}

\end{document}